\documentclass[reqno0,fleqn,12pt]{amsart}
\pdfoutput=1
\usepackage{dsfont}
\usepackage{bbm}
\usepackage[nodate]{datetime}
\usepackage[hmargin=30mm,top=25mm,bottom=25mm,a4paper]{geometry}
\usepackage{color}
\usepackage{subfig}
\usepackage{fancyhdr}
\usepackage{amsmath,amssymb,amsfonts,amsthm}
\usepackage{amstext}
\usepackage{amsmath}
\usepackage{amssymb}
\usepackage{enumerate}
\usepackage{amsbsy}
\usepackage{amsopn}
\usepackage{bbm,amsthm}
\usepackage{amscd}
\usepackage{amsxtra}
\usepackage{todonotes}

\newtheorem{theorem}{Theorem}[section]
\newtheorem*{theorem*}{Theorem}
\newtheorem{lemma}{Lemma}[section]

\newtheorem{corollary}{Corollary}[section]

\theoremstyle{remark}

\newcommand{\CC}{\mathds{C}}

\newcommand{\RR}{\mathds{R}}
\newcommand{\NN}{\mathds{N}}
\newcommand{\HH}{\mathds{H}}

\newcommand{\EE}{\mathbb{E}}

\begin{document}
\title{A note on prime number races and zero free regions for $L$ functions}
\author{Marco Aymone}
\begin{abstract}
Let $\chi$ be a real and non-principal Dirichlet character, $L(s,\chi)$ its Dirichlet $L$-function and let $p$ be a generic prime number. We prove the following result: If for some $0\leq \sigma<1$ the partial sums $\sum_{p\leq x}\chi(p)p^{-\sigma}$ change sign only for a finite number of integers $x$, then there exists $\epsilon>0$ such that $L(s,\chi)$ has no zeros in the half plane $Re(s)>1-\epsilon$. 
\end{abstract}

\maketitle

\section{Introduction.}
Let $\chi$ be a non-principal Dirichlet character and $L(s,\chi)$ its Dirichlet $L$-function.
Many central problems in analytic number theory such as questions about the distribution of primes in arithmetic progressions can be phrased in terms of zero-free regions for $L(s,\chi)$. The typical zero free-region for $L(s,\chi)$ known up to date is: If $q$ is the modulus of $\chi$, then there exists a constant $c>0$ such $L(\sigma+it,\chi)\neq 0$ for all $\sigma$ and $t$ such that
\begin{equation*}
\sigma>1-\frac{c}{\log q(2+|t|)},
\end{equation*}  
with at most one possible exception -- A real zero $\beta<1$ -- in the case that $\chi$ is real (see \cite{montgomerylivro} pg. 360).

Let $p$ be a generic prime number and $\mathcal{P}$ be the set of primes. Let $\chi_4$ be the real and non-principal Dirichlet character mod $4$, \textit{i.e.}, $\chi_4(n)=1$ if $n\equiv 1\mod 4$, $\chi_4(n)=-1$ if $n\equiv 3\mod 4$ and $\chi_4(n)=0$ if $n$ is even. Then the sum $\sum_{p\leq x}\chi_4(p)$ is the prime number race mod $4$: the number of primes up to $x$ of the form $4n+1$ minus the number of primes up to $x$ of the form $4n+3$. 

In 1853, in a letter to Fuss,  it has been observed by Tch\'ebyshev that seems to be more primes of the form $4n+3$ than primes of the form $4n+1$. In other words, it seems that $\sum_{p\leq x}\chi_4(p)$ is negative for most values of $x$. This observation led to many investigations on prime number races for a generic modulus $q$. For an historical background on prime number races we refer reader to the expository paper of Granville and Martin \cite{granvilleraces}, and for recent results in this topic we refer to the paper of Harper and Lamzouri \cite{harperraces} and the references therein.  

In the prime number race mod $4$, the partial sums $\sum_{p\leq x}\chi_4(p)$ change sign for an infinite number of integers $x$. However, for $0<\sigma<1$, it is possible that the \textit{weighted} prime number race $\sum_{p\leq x}\frac{\chi_4(p)}{p^\sigma}$ change sign only for a finite number of integers $x$. Indeed, if we assume the Riemann Hypothesis for $L(s,\chi_4)$, we have that for some constant $c>0$ (see the concluding remarks below), for fixed $\sigma\to1/2^+$ 
\begin{equation}\label{equacao auxiliar 0}
\sum_{p\leq x}\frac{\chi_4(p)}{p^\sigma}\leq-\frac{1}{2}\log\left( \frac{1}{2\sigma-1} \right)+c,
\end{equation} 
for all $x$ greater than an $x_0=x_0(\sigma)$, and since the right side of \eqref{equacao auxiliar 0} becomes negative and blows as $\sigma\to1/2^+$, we might say that this is somehow in agreement with the intuiton behind the Tch\'ebyshev bias.

In particular, if \eqref{equacao auxiliar 0} holds for all $x\geq x_0$, the partial sums $\sum_{p\leq x}\frac{\chi_4(p)}{p^\sigma}$ change sign only for a finite number of integers $x$. If this last assertion is true, our first result states:

\begin{theorem}\label{teorema 1} Let $\chi$ be a real and non-principal Dirichlet character. If for some $0\leq\sigma<1$ the partial sums $\sum_{p\leq x}\chi(p)p^{-\sigma}$ change sign only for a finite number of integers $x\geq 1$, then there exists $\epsilon>0$ such that $L(s,\chi)\neq 0$ for all $s$ in the half plane
$Re(s)>1-\epsilon$. 
\end{theorem}
Let $\mathcal{P}$ be the set of prime numbers.
\begin{corollary} Under the hypothesis of Theorem \ref{teorema 1}, we have that for some $\epsilon>0$, $\sum_{p\in\mathcal{P}}\chi(p)p^{-(1-\epsilon)}$ converges, and hence, the Euler product formula
\begin{equation*}
L(s,\chi)=\prod_{p\in\mathcal{P}}\bigg{(}1-\frac{\chi(p)}{p^s}\bigg{)}^{-1}
\end{equation*}
holds for all $Re(s)>1-\epsilon$.
\end{corollary}

It is worth mentioning that a converse result holds for Theorem 1.1:
\begin{theorem}\label{teorema 2} Let $\chi$ be a real and non-principal Dirichlet character. If for some $\epsilon>0$ we have that $L(s,\chi)\neq 0$ for all $Re(s)>1-\epsilon$, then there exists $1-\epsilon<\sigma<1$
such that $\sum_{p\leq x }\chi(p)p^{-\sigma}$ change sign only for a finite number of integers $x\geq 1$.
\end{theorem}

The proof of Theorem \ref{teorema 1} is an application of an integral version of Landau's oscillation Theorem: If $A:[0,\infty)\to\RR$ is a bounded Riemann-integrable function in any finite interval $[1,x]$, and such that $A(x)\geq 0$ for all $x\geq x_0>1$, then the function 
\begin{equation*}
F(s)=\int_{1}^\infty\frac{A(x)}{x^s}dx
\end{equation*} 
has a singularity in its abcissa of convergence.

In fact, the proof of Theorem \ref{teorema 1} is done by the following steps: For $Re(s)>1$ we can write
\begin{equation*}
\sum_{p\in\mathcal{P}}\frac{\chi(p)}{p^s}=(s-\sigma)\int_{1}^\infty \frac{\sum_{p\leq x}\frac{\chi(p)}{p^\sigma}}{x^{s+1-\sigma}}dx.
\end{equation*}
For $\chi$ non-principal, $L(1,\chi)\neq 0$ and since $L(s,\chi)$ is analytic in $Re(s)>0$, there exists an open ball $B$ of center $1$ and radius $\delta>0$ in which $L(s,\chi)\neq 0$. The union of the half plane $Re(s)>1$ with this open ball is a simply connected domain, and since $L(s,\chi)\neq0$ in this domain, there exists a branch of the Logarithm for $L(s,\chi)$. The existence of this branch implies that $\int_{1}^\infty \frac{\sum_{p\leq x}\frac{\chi(p)}{p^\sigma}}{x^{s+1-\sigma}}dx$ is analytic at a neighborhood of $s=1$, and hence, by the Landau's oscillation Theorem, this integral converges for $s=1-\epsilon$, for some $\epsilon>0$. 

If $f:\NN\to[-1,1]$ is a completely multiplicative function that it is \textit{small on average}, \textit{i.e.}, $\sum_{n\leq x}f(n)=o(x^{1-\delta})$ for some $\delta>0$, then the Dirichlet series $F(s):=\sum_{n=1}^\infty f(n)n^{-s}$ is analytic in $Re(s)>1-\delta$. In \cite{koukou}, Koukoulopoulos proved that if $f$ is small on average and if $F(1)\neq 0$, then $\sum_{p\leq x}f(p)\log p \ll x \exp(-c\sqrt{\log x})$, for some constant $c>0$. Let $\mathcal{P}$ be the set of primes. In \cite{aymonebiased} it has been proved that under biased assumptions, \textit{i.e.}, if at primes $(f(p))_{p\in\mathcal{P}}$ is a sequence of independent random variables such that $\EE f(p)<0$ for all primes $p$, then the assumptions that $f$ is small on average almost surely (\textit{a.s.}) and $F(1)\neq 0$ \textit{a.s.} imply that $\sum_{p\in\mathcal{P}}f(p)p^{-(1-\epsilon)}$ converges for some $\epsilon>0$ \textit{a.s.}, and hence that $F(s)\neq 0$ for all $Re(s)>1-\epsilon$, \textit{a.s}.

The same lines of the proof of Theorems \ref{teorema 1} and \ref{teorema 2} allow us to show that:
\begin{theorem} If $f:\NN\to[-1,1]$ is a completely multiplicative function that is \textit{small on average},  and if the Dirichlet series of $f$, say $F(s)$, is such that $F(1)\neq0$, then there exists $\epsilon>0$ such that $F(s)\neq 0$ for all $Re(s)>1-\epsilon$ if and only if there exists $0\leq \sigma<1$ such that the partial sums $\sum_{p\leq x}f(p)p^{-\sigma}$ change sign only for a finite number of integers $x$.
\end{theorem} 
\noindent \textbf{Acknowledgements.} I would like to thank the anonymous referee for a careful reading the paper and for useful comments.

\section{Proof of the main results}
\subsection*{Notation} Here $\chi$ is a Dirichlet character and $L(s,\chi)=\sum_{n=1}^\infty\frac{\chi(n)}{n^s}$. We use both $f(x)\ll g(x)$ and $f(x)=O(g(x))$ whenever there exists a constant $C>0$ such that for all large $x>0$ we have that $|f(x)|\leq C|g(x)|$. Further, $\ll_\delta$ means that the implicit constant may depend on $\delta$. We let $\mathcal{P}$ for the set of primes and $p$ for a generic element of $\mathcal{P}$. For a real number $a$, we denote the half plane $\{s\in\CC:Re(s)>a\}$ by $\HH_a$.
\begin{lemma}\label{lemma log} Let $\chi$ be a real and non-principal Dirichlet character. Then there exists
an analytic function $B:\HH_{1/2}\to\CC$ such that for $a>1/2$, $B(s)\ll_{a}1$ in the half plane $\HH_a$, and for $s\in\HH_1$:
\begin{equation*}
\log L(s,\chi)=\sum_{p\in\mathcal{P}}\frac{\chi(p)}{p^s}+B(s).
\end{equation*}
\end{lemma}
\begin{proof}
This follows from the Euler product formula valid for $s\in\HH_1$:
\begin{equation*}
L(s,\chi)=\prod_{p\in\mathcal{P}}\bigg{(}1-\frac{\chi(p)}{p^s}\bigg{)}^{-1}.
\end{equation*}
Thus 
\begin{align*}
\log L(s,\chi)=\sum_{p\in\mathcal{P}}\log\bigg{(}1-\frac{\chi(p)}{p^s}\bigg{)}^{-1}=\sum_{p\in\mathcal{P}}\sum_{m=1}^\infty\frac{\chi(p)^m}{mp^{ms}}
\end{align*}
where in the last equality above we used the Taylor expansion for each term $\log\big{(}1-\frac{\chi(p)}{p^s}\big{)}^{-1}$. Now, we split this double infinite sum into two infinite sums:
\begin{equation*}
\log L(s,\chi)=\sum_{p\in\mathcal{P}}\frac{\chi(p)}{p^s}+\sum_{p\in\mathcal{P}}\sum_{m=2}^\infty\frac{\chi(p)^m}{mp^{ms}}. 
\end{equation*}
Let $B(s):=\sum_{p\in\mathcal{P}}\sum_{m=2}^\infty\frac{\chi(p)^m}{mp^{ms}}$. Then the inner sum $\sum_m$ is $\ll \frac{1}{p^{2Re(s)}}$. Thus $B(s)$ converges absolutely for each $s\in\HH_{1/2}$, and hence, it defines an analytic function in this half plane. Moreover,  for each fixed $a>1/2$, $B(s)\ll \sum_{p\in\mathcal{P}}\frac{1}{p^{2a}}$.
\end{proof}
\begin{lemma}[Lemma 15.1 of \cite{montgomerylivro}, Landau's oscillation Theorem]\label{lemma Landau} Let $A:[0,\infty)\to\RR$ be a bounded Riemann-integrable function in any finite interval $[1,x]$, and assume that for some large $x_0>0$ we have that $A(x)\geq 0$ for all $x\geq x_0>0$. Let $\sigma_c$ be the infimum of those $\sigma$ for which $\int_{1}^\infty \frac{|A(x)|}{x^\sigma}dx<\infty$. Then the function 
\begin{equation*}
F(s)=\int_{1}^\infty\frac{A(x)}{x^s}dx
\end{equation*} 
is analytic in $\HH_{\sigma_c}$ and has a singularity at $\sigma_c$.
\end{lemma}

\begin{lemma}[Corollary 6.17 of \cite{conway}]\label{lemma branch} Let $G$ be a simply connected domain
and $f:G\to\CC$ an analytic function such that $f(s)\neq 0$ for all $s\in G$. Then there exists an analytic function $g:G\to\CC$ such that $f(z)=\exp(g(z))$. If $w:G\to\CC$ is another analytic function such that $f(s)=\exp(w(s))$ for all $s\in G$, then there exists $c\in\CC$ such that $g(s)-w(s)=c$, for all $s\in G$. 
\end{lemma}
\begin{proof}[Proof of Theorem \ref{teorema 1}]
Let $A(x)=\sum_{p\leq x}\frac{\chi(p)}{p^\sigma}$, $0\leq\sigma<1$. Assume that for some $x_0>0$, $A(x)$ is either $A(x)\geq 0$ for all $x\geq x_0$ or $A(x)\leq 0$ for all $x\geq x_0$. Clearly $A(x)$ is a bounded Riemann-integrable function in any finite interval $[1,x]$. Let $s\in \HH_1$. Then
\begin{align*}
\sum_{p\in\mathcal{P}}\frac{\chi(p)}{p^s}=\sum_{p\in\mathcal{P}}\frac{\chi(p)}{p^\sigma}\frac{1}{p^{s-\sigma}}=\int_1^\infty \frac{1}{u^{s-\sigma}}dA(u)=(s-\sigma)\int_{1}^\infty \frac{A(u)}{u^{s+1-\sigma}}du. 
\end{align*}
Since the partial sums $\sum_{n\leq x} \chi(n)\ll 1$, we have that $L(s,\chi)$ converges for all $s\in\HH_0$, and hence, it is analytic in this half plane. Further, by Lemma \ref{lemma log}, we have that $L(s,\chi)\neq 0$ for $s\in\HH_1$. Moreover, if $\chi$ is non-principal, $L(1,\chi)\neq 0$. Thus there exists an open ball $B$ with positive radius and centered at $s=1$ such that $L(s,\chi)\neq 0$ for all $s\in B$. It follows that $L(s,\chi)\neq 0$ for all $s\in \HH_1\cup B$. The set $\HH_1\cup B$ is simply connected. Thus, by Lemma \ref{lemma branch}, there exists an analytic function $\log^*L(\cdot,\chi):\HH_1\cup B\to\CC$
such that $L(s,\chi)=\exp(\log^*L(s,\chi))$ for all $s\in\HH_1\cup B$. By Lemma \ref{lemma log}, we have for $s\in\HH_1$  
\begin{equation*}
L(s,\chi)=\exp\bigg{(}\sum_{p\in\mathcal{P}}\frac{\chi(p)}{p^s}+B(s)\bigg{)}.
\end{equation*}
Since $\HH_1$ also is simply connected, by Lemma \ref{lemma branch} it follows that there exists a constant $c\in\CC$ such that for all $s\in\HH_1$
\begin{equation*}
\log^*L(s,\chi)=(s-\sigma)\int_{1}^\infty \frac{A(u)}{u^{s+1-\sigma}}du+B(s)+c.
\end{equation*} 
Thus:
\begin{equation*}
\int_{1}^\infty \frac{A(u)}{u^{s+1-\sigma}}du=\frac{\log^*L(s,\chi)-B(s)-c}{s-\sigma}.
\end{equation*}
It follows that $\int_{1}^\infty \frac{A(u)}{u^{s+1-\sigma}}du$ has an analytic continuation to $\HH_1\cup B'$, where $B'$ is an open ball of positive radius and centered at $1$. Since $\HH_1\cup B' $ is a simply connected domain, this analytic continuation is unique. Thus, the function defined by this integral for $Re(s)>1$ and by the analytic continuation for the other values of $s$ is analytic in $\HH_1\cup B' $, in particular, this function does not have a singularity $s=1$. Hence, by Landau's oscillation Theorem (Lemma \ref{lemma Landau}), $\sigma=1$ can not be the abscissa of convergence of the integral: we have that $\int_{1}^\infty \frac{A(u)}{u^{s+1-\sigma}}du$ converges for $s=1-\epsilon$, for some $\epsilon>0$. Since $A(x)$ changes sign only for a finite number of $x\geq 1$, this convergence is absolute, and hence $\int_{1}^\infty \frac{A(u)}{u^{s+1-\sigma}}du$ is an analytic function in $\HH_{1-\epsilon}$. It follows that $\log L(s,\chi)$ has an analytic continuation to $\HH_{1-\epsilon}$, and hence, $L(s,\chi)\neq 0$ for $s\in\HH_{1-\epsilon}$.     \end{proof}  

Now we will prove Corollary 1.1, but before that, we will require the following results for Dirichlet series which we refer to Theorem 15, pg. 119,  and to Theorem 4, pg. 134, of the book of Tenenbaum \cite{tenenbaumlivro}:

\begin{lemma}\label{lemma resultados auxiliares para series de Dirichlet} 
Result 1) Let $F(s)=\sum_{n=1}^\infty\frac{a_n}{n^s}$ be a Dirichlet series with finite abcsissa of convergence $\sigma_c$. Let $\sigma_0>\sigma_c$. Then uniformly for $\sigma_0\leq \sigma\leq \sigma_c+1$ we have that $F(\sigma+it)\ll |t|^{1-(\sigma-\sigma_c)+\delta}$. \\
Result 2) If $F(s)=\sum_{n=1}^\infty\frac{a_n}{n^s}$ has a finite abscissa of convergence and if $\sigma_0$ is some real number for which $F(s)$ has an analytic continuation to $\HH_{\sigma_0}$ satisfying, for each $\sigma>\sigma_0$, $F(\sigma+it)\ll t^\delta$, for all $\delta>0$, then $\sum_{n=1}^\infty\frac{a_n}{n^s}$ converges for all $s\in\HH_{\sigma_0}$.

\end{lemma}

\begin{proof}[Proof of Corollary 1.1] In Lemma \ref{lemma log} we proved that for $Re(s)>1$ 
\begin{equation}\label{equacao logL}
\log L(s,\chi)=\sum_{p\in\mathcal{P}}\frac{\chi(p)}{p^s}+\sum_{p\in\mathcal{P}}\sum_{m=2}^\infty\frac{\chi(p)^m}{mp^{ms}},
\end{equation}
where the infinite double sum above converges absolutely for all $s\in\HH_{1/2}$. Firstly, we recall the principle of analytic continuation which says that when two analytic functions defined in the same open connected set $G$ coincide in a compact and infinite set $K\subset G$, then these two functions are the same. On the one hand, if $L(s,\chi)\neq 0$ for all $s\in \HH_{1-\epsilon}$, then the function $\log L(s,\chi)$ is well defined and analytic in $\HH_{1-\epsilon}$, and we can make it to coincide with the right side of \eqref{equacao logL} for all $s\in\HH_1$. Indeed, by Lemma \ref{lemma branch}, as $\HH_{1-\epsilon}$ is simply connected, there is an analytic function $\log^*L(s,\chi)$ defined for all $s$ in this half plane such that $L(s,\chi)=\exp(\log^*L(s,\chi))$. Since $\HH_1$ also is simply connected, by Lemma \ref{lemma branch} again we have that $\log^*L(s,\chi)$ minus the right side of \eqref{equacao logL} is a constant $c$, for all $s\in\HH_1$. Thus, by making a little abuse of notation and defining $\log L(s,\chi)=\log^*L(s,\chi)-c$ for all $s\in\HH_{1-\epsilon}$, we have that $\log L(s,\chi)$ is analytic in $\HH_{1-\epsilon}$, coincides with the right side of \eqref{equacao logL} in $\HH_1$, and in $\HH_1$, $L(s,\chi)=\exp(\log L(s,\chi))$, and hence, by the principle of analytic continuation stated above this last relation holds for all $s\in\HH_{1-\epsilon}$.  Thus, we only need to show that the series $\sum_{p\in\mathcal{P}}\frac{\chi(p)}{p^s}$ converges for all $s\in\HH_{1-\epsilon}$ because, in this case, the function in the right side of \eqref{equacao logL} will be analytic in $\HH_{1-\epsilon}$, and since coincides with $\log L(s,\chi)$ for all $s\in\HH_1$, the identity \eqref{equacao logL} will hold for all $s\in\HH_{1-\epsilon}$. Therefore, by undoing the Taylor approximation, we will have for $s\in\HH_{1-\epsilon}$:
\begin{align*}
L(s,\chi)&=\exp(\log L(s,\chi))=\exp\left(\sum_{p\in\mathcal{P}}\frac{\chi(p)}{p^s}+\sum_{p\in\mathcal{P}}\sum_{m=2}^\infty\frac{\chi(p)^m}{mp^{ms}}\right)\\
&=\exp\left(\sum_{p\in\mathcal{P}}\sum_{m=1}^\infty\frac{\chi(p)^m}{mp^{ms}}\right)=\exp\left(\sum_{p\in\mathcal{P}} \log\left(1-\frac{\chi(p)}{p^s}\right)^{-1}  \right)\\
&=\prod_{p\in\mathcal{P}}\left(1-\frac{\chi(p)}{p^s}\right)^{-1}.
\end{align*}

By Lemma \ref{lemma resultados auxiliares para series de Dirichlet} result 1), uniformly for $\sigma_0\leq \sigma\leq \sigma_c+1$, we have that $F(\sigma+it)\ll |t|^{1-(\sigma-\sigma_c)+\delta}$. Since $L(s,\chi)$ is convergent for $s\in\HH_0$, we have for some constant $A>0$, $L(\sigma+it,\chi)\ll|t|^A$.

On the hypothesis of Theorem \ref{teorema 1}, we have that $L(s,\chi)\neq 0$ for $s\in\HH_{1-\epsilon}$.
Hence, for $\sigma>1-\epsilon$, $\log|L(\sigma+it,\chi)|\ll A\log (|t|+2)$. By applying the Borel-Caratheodory theorem, we can conclude, in the same line of reasoning of Theorem 14.2 of \cite{tit} that
$\log L(\sigma+it,\chi)\ll \log(|t|+2)$. Thus, by Lemma \ref{lemma log} and Lemma \ref{lemma branch} we have that $\sum_{p\in\mathcal{P}}\frac{\chi(p)}{p^s}$ has an analytic continuation to $\HH_{1-\epsilon}$ given by $F(s)=\log L(s,\chi)-B(s)+c$, for some constant $c$. This analytic continuation is, for $\sigma>1-\epsilon$, $F(\sigma+it)\ll \log(|t|+2)\ll t^\delta$, for all $\delta>0$. 

Finally, by Lemma \ref{lemma resultados auxiliares para series de Dirichlet} result 2), we conclude that  $\sum_{p\in\mathcal{P}}\frac{\chi(p)}{p^s}$ converges for all $s\in\HH_{1-\epsilon}$.  \end{proof}

\begin{proof}[Proof of Theorem \ref{teorema 2}] In the proof of Corollary 1.1, we showed that the hypothesis $L(s,\chi)\neq 0$ for all $s\in\HH_{1-\epsilon}$ implies that the series $\sum_{p\in\mathcal{P}}\frac{\chi(p)}{p^s}$ converges for all $s\in\HH_{1-\epsilon}$. We claim that there exists $\sigma\in(1-\epsilon,1)$ for which $\sum_{p\in\mathcal{P}}\frac{\chi(p)}{p^\sigma}\neq 0$. By contradiction, if no such $\sigma$ exists, then $\sum_{p\in\mathcal{P}}\frac{\chi(p)}{p^s}=0$ for all $s\in (1-\epsilon, 1)$, and since this Dirichlet series is an analytic function, it follows that this analytic function is equal to zero everywhere. Hence, by Theorem 1.6 of \cite{montgomerylivro}, we have that $\chi(p)=0$ for all $p\in\mathcal{P}$, which is a contradiction. Hence, there exists $\sigma\in(1-\epsilon,1)$ such that the partial sums $\sum_{p\leq x}\frac{\chi(p)}{p^\sigma}$ converges, as $x\to\infty$, to a non-zero value. Hence, this partial sums can change sign only for a finite number of integers $x$. \end{proof}

\section{Concluding remarks}
In this section we discuss the condition in which $\sum_{p\leq x}\frac{\chi(p)}{p^\sigma}$ change sign only for a finite number of integers $x$, for some $0\leq \sigma<1$.

We begin by justifying \eqref{equacao auxiliar 0}. If the Riemann Hypothesis for $L(s,\chi)$ is true for a real and non-principal Dirichlet character $\chi$, by the proof of Corollary 1.1 we see that for each $\sigma>1/2$
\begin{equation*}
\log L(\sigma,\chi)=\sum_{p\in\mathcal{P}}\sum_{m=1}^\infty\frac{\chi(p)^m}{mp^{m\sigma}}=\sum_{p\in\mathcal{P}}\frac{\chi(p)}{p^\sigma}+\frac{1}{2}\sum_{p\in\mathcal{P}}\frac{\chi(p)^2}{p^{2\sigma}}+\sum_{p\in\mathcal{P}}\sum_{m=3}^\infty\frac{\chi(p)^m}{mp^{m\sigma}}.
\end{equation*}
The last infinite double sum above converges absolutely for $\sigma>1/3$, and hence, it is $O(1)$ in the interval $\sigma\in[1/2,1]$. Further, since $\chi$ is real, we have that $\chi^2(p)=1$, except for a finite quantity of primes $p$. Now, for $2\sigma>1$, the Riemann $\zeta$ function satisfies $\zeta(2\sigma)=\frac{1}{2\sigma-1}+O(1)$ and $\log\zeta(2\sigma)=\sum_{p\in\mathcal{P}}\frac{1}{p^{2\sigma}}+O(1)$. This allow us to write
\begin{equation}\label{equacao auxiliar 1}
\log L(\sigma,\chi)=\sum_{p\in\mathcal{P}}\frac{\chi(p)}{p^\sigma}+\frac{1}{2}\log\left(\frac{1}{2\sigma-1}\right)+O(1).
\end{equation}

Since $L(s,\chi)$ is analytic in $\HH_0$, we have that $L(\sigma,\chi)\leq C$, for all $\sigma\in[1/2,1]$, where $C>0$ is a finite constant. 
Plugging this into \eqref{equacao auxiliar 1}, we obtain that for some constant $D>0$, for all $\sigma\in(1/2,1]$
\begin{equation*}
\sum_{p\in\mathcal{P}}\frac{\chi(p)}{p^\sigma}\leq -\frac{1}{2}\log\left(\frac{1}{2\sigma-1}\right)+D.
\end{equation*}

Since the series in the left side above converges, we have that for sufficiently large $x_0=x_0(\sigma)>0$ depending only on $\sigma$, for all $x\geq x_0$
\begin{equation}\label{equacao auxiliar 2}
\sum_{p\leq x}\frac{\chi(p)}{p^\sigma}\leq -\frac{1}{2}\log\left(\frac{1}{2\sigma-1}\right)+D+1,
\end{equation}
in particular $\sum_{p\leq x}\frac{\chi(p)}{p^\sigma}$ change sign only for a finite number of integers $x\geq 1$. Since the right side of \eqref{equacao auxiliar 2} blows to $-\infty$ as $\sigma\to1/2^+$, we conjecture:

\noindent \textbf{Conjecture}. For the prime number race mod $4$, $\chi_4$ being the non-principal Dirichlet character mod $4$, $\lim_{x\to\infty}\sum_{p\leq x}\frac{\chi_4(p)}{\sqrt{p}}=-\infty$.

Summarizing, the Riemann Hypothesis for $L(s,\chi)$ implies \eqref{equacao auxiliar 2} for all $x\geq x_0$ which is in agreement with the intuition behind the Tch\'ebyshev bias mod $4$. On the other hand, \eqref{equacao auxiliar 2} implies that for some $\sigma\in[0,1)$ the partial sums $\sum_{p\leq x}\frac{\chi(p)}{p^\sigma}$ change sign only for a finite number of integers $x$, which in turn, by Theorem \ref{teorema 1}, does not implies the Riemann Hypothesis for $L(s,\chi)$, but it gives that a certain half plane is a zero free region for it. Thus, our hypothesis is almost of same strenght of the Riemann Hypothesis for $L(s,\chi)$.

{\small{\sc \noindent Marco Aymone \\
Departamento de Matem\'atica, Universidade Federal de Minas Gerais, Av. Ant\^onio Carlos, 6627, CEP 31270-901, Belo Horizonte, MG, Brazil.} \\
\textit{Email address:} marco@mat.ufmg.br}
\vspace{0.5cm}

\end{document}